\documentclass[11pt]{amsart}
\usepackage{amscd,amsmath,amssymb,amsthm,latexsym,esint,cite}
\usepackage{verbatim}
\usepackage[left=2.65cm,right=2.65cm,top=2.8cm,bottom=2.8cm]{geometry}

\usepackage{color,enumitem,graphicx}
\usepackage[colorlinks=true,urlcolor=blue, citecolor=red,linkcolor=blue,
linktocpage,pdfpagelabels,bookmarksnumbered,bookmarksopen]{hyperref}

\usepackage[hyperpageref]{backref}
\usepackage[english]{babel}

\newcommand{\abs}[1]{\left|#1\right|}

\newcommand{\bdry}[1]{\partial #1}

\newcommand{\bgset}[1]{\big\{#1\big\}}

\newcommand{\A}{{\mathcal A}}

\newcommand{\F}{{\mathcal F}}

\newcommand{\dint}{\ds{\int}}

\newcommand{\ds}[1]{\displaystyle #1}

\newcommand{\eps}{\varepsilon}

\renewcommand{\L}{{\mathcal L}}
\newcommand{\M}{{\mathcal M}}
\newcommand{\N}{\mathbb N}

\newcommand{\norm}[2][]{\left\|#2\right\|_{#1}}
\renewcommand{\O}{\mathcal{O}}
\renewcommand{\o}{\text{o}}
\newcommand{\PS}[1]{$(\text{PS})_{#1}$}

\newcommand{\pnorm}[2][]{\if #1'' \left|#2\right|_p \else \left|#2\right|_{#1} \fi}

\newcommand{\R}{\mathbb R}

\newcommand{\seq}[1]{\left(#1\right)}
\newcommand{\set}[1]{\left\{#1\right\}}

\newcommand{\spnorm}[2][]{\if #1'' |#2|_p \else |#2|_{#1} \fi}

\newcommand{\vol}[1]{\if #1'' \L \else \L(#1) \fi}

\newcommand{\wto}{\rightharpoonup}
\newcommand{\Z}{\mathbb Z}

\newtheorem{lemma}{Lemma}[section]

\newtheorem{theorem}[lemma]{Theorem}
\newtheorem{corollary}{Corollary}[section]

\theoremstyle{definition}

\theoremstyle{remark}

\newtheorem{remark}[lemma]{Remark}
\theoremstyle{definition}

\newenvironment{enumroman}{\begin{enumerate}
		
		}{\end{enumerate}}

\numberwithin{equation}{section}

\title[Problems with fractional Trudinger-Moser nonlinearity]{Bifurcation results for problems with \\ fractional Trudinger-Moser nonlinearity}

\author[K.\ Perera]{Kanishka Perera}
\author[M.\ Squassina]{Marco Squassina}

\address[K. Perera]{Department of Mathematical Sciences
	\newline\indent
	Florida Institute of Technology
	\newline\indent
	150 W University Blvd, Melbourne, FL 32901, USA}
\email{kperera@fit.edu}

\address[M.\ Squassina]{Dipartimento di Matematica e Fisica \newline\indent
	Universit\`a Cattolica del Sacro Cuore \newline\indent
	Via dei Musei 41, 25121 Brescia, Italy}
\email{marco.squassina@unicatt.it}

\thanks{The second author is member of {Gruppo Nazionale per l'Analisi 
		Ma\-te\-ma\-ti\-ca, la Probabilit\`a e le loro Applicazioni} (GNAMPA) of the {\em Istituto Nazionale di Alta Matematica} (INdAM)}


\begin{document}

\begin{abstract}
By using a suitable topological argument based on cohomological linking and by exploiting a Trudinger--Moser inequality
in fractional spaces recently obtained, we  prove existence of multiple solutions for a problem involving 
the nonlinear fractional laplacian and a related critical exponential nonlinearity. This extends results in the literature for the $N$-Laplacian operator.
\end{abstract}

\maketitle

%\begin{center}
%	\begin{minipage}{12cm}
%		\small
%		\tableofcontents
%	\end{minipage}
%\end{center}
%
%\medskip

\section{Introduction}
\subsection{Overview}
Let $\Omega$ be a bounded domain in $\R^N$ with $N \ge 2$ and with Lipschitz boundary $\bdry{\Omega}$. 
We denote by $\omega_{N-1}$ the measure of the unit sphere in $\R^N$ and $N'=N/(N-1)$. Since the time when the 
Trudinger-Moser inequality was first proved (cf.\ \cite{trud,moser,carleson})
$$
\sup_{u\in W^{1,N}_0(\Omega),\,\,\|\nabla u\|_N\leq 1}\int_{\Omega} e^{\alpha_N |u|^{N'}}dx<+\infty,\qquad
\alpha_N=N\omega_{N-1}^{1/(N-1)},
$$ 
existence and multiplicity of solutions for various nonlinear problems with exponential 
nonlinearity were investigated. For instance, Adimurthi \cite{adimurt} proved the existence of a positive 
solution to the quasi-linear elliptic problem
\begin{equation} \label{1-local}
	\left\{\begin{aligned}
		- \Delta_N u & = \lambda\, |u|^{N-2}\, u\, e^{\, |u|^{N'}} && \text{in } \Omega\\
		u & = 0 && \text{on } \partial\Omega,
	\end{aligned}\right.
\end{equation}
where $\Delta_N u:={\rm div}(|\nabla u|^{N-2}\nabla u)$ is the $N$-Laplacian operator for $0<\lambda<\lambda_1(N)$, being 
$\lambda_1(N)>0$ the first eigenvalue of $\Delta_N$ with Dirichlet boundary conditions, see also \cite{doO}. The case $N=2$ was
investigated in \cite{defig1,defig2}, where the existence of a nontrivial solution was found for $\lambda\geq \lambda_1$. Recently,
in \cite{yype} it was proved that problem \eqref{1-local} admits a nontrivial weak solution whenever $\lambda>0$ is {\em not} an eigenvalue 
of $-\Delta_N$ in $\Omega$ with Dirichlet boundary conditions. In addition in \cite{yype} a bifurcation result for higher (nonlinear) eigenvalues
(which are suitably defined via the cohomological index) is also obtained, yielding in  turn multiplicity results.

The issue of Trudinger-Moser type embeddings for fractional spaces is rather delicate and only quite recently, Parini and Ruf \cite{par-ruf} (see also 
the refinement obtained in \cite{iula0})
provided a partial result in the {\em Sobolev-Slobodeckij} space 
$$
W^{s,N/s}_0(\Omega),\quad  \text{$s\in (0,1)$,\, $N\geq 1$},
$$
defined as the completion of $C^\infty_0(\Omega)$ for the norm
\[
\|u\|=[u]_{s,N/s} := \Big(\int_{\R^{2N}} \frac{|u(x) - u(y)|^{N/s}}{|x - y|^{2N}}\, dx dy\Big)^{s/N}.
\]
We also refer the reader to \cite{Oza,KSW,martinaz,iula} for results in a different functional framework, namely
the {\em Bessel potential} spaces $H^{s,p}$.
In fact, they proved that the supremum $\alpha_{N,s}(\Omega)$ of $\alpha \ge 0$ with
\begin{equation}
\label{fracTM}
\sup_{u\in W^{s,N/s}_0(\Omega),\; [u]_{s,N/s} \le 1}\int_{\Omega} e^{\, \alpha\, |u|^{N/(N-s)}}\, dx<+\infty,
\end{equation}
is positive and finite.
Furthermore, they proved the existence of $\alpha_{N,s}^*(\Omega)\geq \alpha_{N,s}(\Omega)$ such that the supremum in \eqref{fracTM} is $+\infty$
for $\alpha>\alpha_{N,s}^*(\Omega)$. On the other hand it still remains {\em unknown} whether 
$$
\alpha_{N,s}(\Omega)=\alpha_{N,s}^*(\Omega).
$$ 
The case $N=1$ and $s=1/2$ was earlier considered in \cite{IaSq} (see also \cite{giacom}), where the authors
study the existence of weak solutions to the problem
	\begin{equation*} 
	-\frac{C_s}{2}\displaystyle\int_{\R}\frac{u(x+y)+u(x-y)-2 u(x)}{|y|^{1+2s}}dy=f(u),\quad\,\,  u\in W_0^{1/2,2}(-1,1),
	\end{equation*}
where $C_s>0$ is a suitable normalization constant. We also mention \cite{JMdO1,JMdO2} for other investigations in the one
dimensional case on the whole space $\R$, facing the problem of the lack of compactness. In particular in \cite{JMdO2}, the existence 
of {\em ground state} solutions for the problem
	\begin{equation*} 
-\frac{C_s}{2}\displaystyle\int_{\R}\frac{u(x+y)+u(x-y)-2 u(x)}{|y|^{1+2s}}dy+u=f(u),\quad\,\,  u\in W_0^{1/2,2}(\R),
\end{equation*}
was proved, where $f$ is a Trudinger–Moser critical growth nonlinearity.

To the authors' knowledge, in the framework of the {\em Sobolev-Slobodeckij} spaces $W^{s,N/s}_0(\Omega)$, fractional counterparts
of the local quasilinear $N$-Laplacian problem \eqref{1-local} were not previously tackled in the literature. This is precisely the goal of this manuscript.

\subsection{The main result}
Let $N\geq 1$ and $s\in (0,1)$. In the following, the standard norm for the $L^p$ space will always be denoted by $|\cdot|_p$.
For $\lambda>0,$ we consider the quasilinear problem
\begin{equation} \label{1}
	\left\{\begin{aligned}
		(- \Delta)_{N/s}^s\, u & = \lambda\, |u|^{(N-2s)/s}\, u\, e^{\, |u|^{N/(N-s)}} && \text{in } \Omega\\
		u & = 0 && \text{in } \R^N \setminus \Omega,
	\end{aligned}\right.
\end{equation}
where $(- \Delta)_{N/s}^s$ is the nonlinear nonlocal operator defined on smooth functions by
\[
(- \Delta)_{N/s}^s\, u(x) := 2 \lim_{\eps \searrow 0} \int_{\R^N \setminus B_\eps(x)} \frac{|u(x) - u(y)|^{(N-2s)/s}\, (u(x) - u(y))}{|x - y|^{2N}}\, dy, \quad x \in \R^N.
\]
We refer the interested reader to \cite{rev-paper} and the references therein for an overview on recent progresses on
existence, nonexistence and regularity results for equations involving the {\em fractional $p$-laplacian} operator $(-\Delta)^s_p$, $p>1$. 
The standard sequence of eigenvalues for $(- \Delta)_{N/s}^s$ via the {\em Krasnoselskii genus} does not furnish enough information on the structure 
of sublevels and thus the eigenvalues will be introduced via the {\em cohomological index}. 
We consider critical values of the functional
\[
\Psi(u) := \frac{1}{\pnorm[N/s]{u}^{N/s}}, \quad u \in \M,\quad \M: = \bgset{u \in W^{s,N/s}_0(\Omega) : \norm{u} = 1}.
\]
Let $\F$ be the class of symmetric sets of $\M$, $i(M)$ the $\Z_2$-cohomological index of 
a $M\subset\F$ and set
\[
\lambda_k := \inf_{\substack{M \in \F\\[1pt] i(M) \ge k}}\, \sup_{u \in M}\, \Psi(u), \quad k \ge 1,
\qquad (\lambda_k\to+\infty)
\]
Consider also the positive constant
\[
\mu_{N,s}(\Omega):= \alpha_{N,s}(\Omega)^{(N-s)/N}\! \left(\frac{N}{s\, \vol{\Omega}}\right)^{s/N},
\]
being $\vol{}$ the Lebesgue measure in $\R^N$. The following is our main result
\begin{theorem}
	Assume that $\lambda_k \le \lambda < \lambda_{k+1} = \cdots = \lambda_{k+m} < \lambda_{k+m+1}$ for some $k, m \ge 1$ and
	\begin{equation*}
	\lambda + \mu_{N,s}(\Omega)\, \lambda^{(N-s)/N} > \lambda_{k+1},
	\end{equation*}
	then problem \eqref{1} has $m$ distinct pairs of nontrivial solutions $\pm\, u^\lambda_j,\, j = 1,\dots,m$ such that $u^\lambda_j \to 0$ as $\lambda \nearrow \lambda_{k+1}$. In particular, if
	\[
	\lambda_k \le \lambda < \lambda_{k+1} < \lambda + \mu_{N,s}(\Omega)\, \lambda^{(N-s)/N}
	\]
	for some $k \ge 1$, then problem \eqref{1} has a nontrivial solution.
\end{theorem}

This result, which follows from the results in Section~\ref{bif-sect}, is nontrivial since 
the classical linking arguments of \cite{defig1,defig2} cannot be used in the quasi-linear setting.
Instead the abstract machinery developed in \cite{yype} will be applied.
We also would like to stress that, since the Trudinger-Moser embedding \eqref{fracTM} still holds with nonoptimal exponent (contrary to the local case),
it is not clear how to prove Brezis-Nirenberg type results, namely that  problem~\eqref{1} 
admits a nontrivial weak solution whenever $\lambda>0$ is not an eigenvalue of $(- \Delta)_{N/s}^s$.

\section{Preliminaries}
As anticipated in the introduction, we work in the fractional Sobolev space $W^{s,N/s}_0(\Omega)$, 
defined as the completion of $C^\infty_0(\Omega)$ with respect to the Gagliardo seminorm
\[
[u]_{s,N/s} = \left(\int_{\R^{2N}} \frac{|u(x) - u(y)|^{N/s}}{|x - y|^{2N}}\, dx dy\right)^{s/N}.
\]
Furthermore, since $\bdry{\Omega}$ is assumed to be Lipschitz, we have (cf.\ \cite[Proposition B.1]{stability})
\[
W^{s,N/s}_0(\Omega) = \set{u \in L^{N/s}(\R^N) : [u]_{s,N/s} < \infty,\; u = 0 \text{ a.e.\! in } \R^N \setminus \Omega}.
\]
A function $u \in W^{s,N/s}_0(\Omega)$ is a {\em weak solution} of problem \eqref{1} if
\begin{multline*}
	\int_{\R^{2N}} \frac{|u(x) - u(y)|^{(N-2s)/s}\, (u(x) - u(y))\, (v(x) - v(y))}{|x - y|^{2N}}\, dx dy\\[5pt]
	= \lambda \int_\Omega |u|^{(N-2s)/s}\, u\, e^{\, |u|^{N/(N-s)}}\, v\, dx, \quad \forall v \in W^{s,N/s}_0(\Omega).
\end{multline*}
As proved in \cite[Proposition 2.12]{ibero}, a weak solution turns into a {\em poinwise} solution
if $u\in C^{1,\gamma}_{{\rm loc}}$ for some $\gamma\in(0,1)$ 
sufficiently close to 1.
The integral on the right-hand side is well-defined in view of 
\cite[Proposition 3.2]{par-ruf} and the H\"older inequality. 
Weak solutions coincide with critical points of the $C^1$ functional 
%$\Phi:W^{s,N/s}_0(\Omega)\to\R$
\[
\Phi(u) = \frac{s}{N} \int_{\R^{2N}} \frac{|u(x) - u(y)|^{N/s}}{|x - y|^{2N}}\, dx dy - \lambda \int_\Omega F(u)\, dx, \quad u \in W^{s,N/s}_0(\Omega),
\]
where $F(t) = \int_0^t f(\tau)\, d\tau$ and $f(t) = |t|^{(N-2s)/s}\, t\, e^{\, |t|^{N/(N-s)}}$.

\noindent
We recall that $W^{s,N/s}_0(\Omega)$ is uniformly convex, and hence reflexive. Indeed, for $u \in W^{s,N/s}_0(\Omega)$, let
\[
\widetilde{u}(x,y) := \frac{u(x) - u(y)}{|x - y|^{2s}},\quad (x,y)\in\R^{2N}.
\]
Then the mapping $u \mapsto \widetilde{u}$ is a linear isometry from $W^{s,N/s}_0(\Omega)$ to $L^{N/s}(\R^{2N})$, so the uniform convexity of $L^{N/s}(\R^{2N})$ gives the conclusion.

We also have the following Br\'ezis-Lieb lemma in $W^{s,N/s}_0(\Omega)$.

\begin{lemma} \label{Lemma 3}
	If $\seq{u_j}$ is bounded in $W^{s,N/s}_0(\Omega)$ and converges to $u$ a.e.\! in $\Omega$, then
	\[
	\norm{u_j}^{N/s} - \norm{u_j - u}^{N/s} \to \norm{u}^{N/s}\!.
	\]
\end{lemma}

\begin{proof}
	Let
	\[
	\widetilde{u}_j(x,y) = \frac{u_j(x) - u_j(y)}{|x - y|^{2s}}, \qquad \widetilde{u}(x,y) = \frac{u(x) - u(y)}{|x - y|^{2s}},
	\]
	and note that $\seq{\widetilde{u}_j}$ is bounded in $L^{N/s}(\R^{2N})$ and converges to $\widetilde{u}$ a.e.\! in $\R^{2N}$. Hence
	\[
	\pnorm[N/s]{\widetilde{u}_j}^{N/s} - \pnorm[N/s]{\widetilde{u}_j - \widetilde{u}}^{N/s} \to \pnorm[N/s]{\widetilde{u}}^{N/s}
	\]
	by the Br\'ezis-Lieb lemma \cite{blieb}, 
	where $\pnorm[N/s]{\cdot}$ denotes the norm in $L^{N/s}(\R^{2N})$, namely the conclusion.
\end{proof}

\noindent
It was shown \cite[Theorem 1.1]{par-ruf} that the supremum $\alpha_{N,s}(\Omega)$ of all $\alpha \ge 0$ such that
\[
\sup \set{\int_\Omega e^{\, \alpha\, |u|^{N/(N-s)}}\, dx : u \in W^{s,N/s}_0(\Omega),\; [u]_{s,N/s} \le 1} < +\infty
\]
satisfies $0 < \alpha_{N,s}(\Omega) < \infty$. The main result of this section is the following theorem, 
which is due to P.L.\ Lions \cite{lio} in the local case $s = 1$.

\begin{theorem} \label{Theorem 1}
	If $\seq{u_j}$ is a sequence in $W^{s,N/s}_0(\Omega)$ with $\norm{u_j} = 1$ for all $j\in\N$ and converging a.e.\! to a nonzero function $u$, then
	\[
	\sup_{j\in\N} \int_\Omega e^{\, \alpha\, |u_j|^{N/(N-s)}}\, dx < +\infty
	\]
	for all $\alpha < \alpha_{N,s}(\Omega)/(1 - \norm{u}^{N/s})^{s/(N-s)}$.
\end{theorem}

\begin{proof}
	We have
	\[
	|u_j|^{N/(N-s)} \le (|u| + |u_j - u|)^{N/(N-s)} \le (p\, |u|)^{N/(N-s)} + (q\, |u_j - u|)^{N/(N-s)},
	\]
	where $1/p + 1/q = 1$. Then
	\[
	\int_\Omega e^{\, \alpha\, |u_j|^{N/(N-s)}}\, dx \le \left(\int_\Omega e^{\, \alpha\, \widetilde{p}\, |u|^{N/(N-s)}}\, dx\right)^{1/p} \left(\int_\Omega e^{\, \alpha\, \widetilde{q}\, |u_j - u|^{N/(N-s)}}\, dx\right)^{1/q}
	\]
	by the H\"{o}lder inequality, where $\widetilde{p} = p^{(2N-s)/(N-s)}$ and $\widetilde{q} = q^{(2N-s)/(N-s)}$. The first integral on the right-hand side is finite, and the second integral equals
	\[
	\int_\Omega e^{\, \alpha\, \widetilde{q}\, \norm{u_j - u}^{N/(N-s)}\, |v_j|^{N/(N-s)}}\, dx,
	\]
	where $v_j = (u_j - u)/\norm{u_j - u}$. By Lemma \ref{Lemma 3}, $\norm{u_j - u}^{N/(N-s)} \to (1 - \norm{u}^{N/s})^{s/(N-s)}$. Taking $q>1$ sufficiently close to $1$, let 
	$$
	\alpha\, \widetilde{q}\, (1 - \norm{u}^{N/s})^{s/(N-s)} < \beta < \alpha_{N,s}(\Omega). 
	$$
	Then $\alpha\, \widetilde{q}\, \norm{u_j - u}^{N/(N-s)} \le \beta$ and hence the last integral is less than or equal to
	\[
	\int_\Omega e^{\, \beta\, |v_j|^{N/(N-s)}}\, dx,
	\]
	for all sufficiently large $j$, which is bounded since $\beta < \alpha_{N,s}(\Omega)$ and $\norm{v_j} = 1$.
\end{proof}

We close this preliminary section with a technical lemma.

\begin{lemma} \label{Lemma 1}
	For all $t \in \R$,
	\begin{enumroman}
		\item \label{Item 2} $F(t) \le \ds{\frac{N - s}{N}\, \frac{t f(t)}{|t|^{N/(N-s)}}}$,
		\item \label{Item 1} $F(t) \le F(1) + \dfrac{s\, (N - s)}{N^2}\; t f(t)$,
		\item \label{Item 3} $\ds{\frac{s}{N}\; t f(t) - F(t) \ge \frac{s^2}{N^2}\, |t|^{N^2/s(N-s)}}$, in particular, $t f(t) \ge \dfrac{N}{s}\, F(t)$,
		\item \label{Item 4} $F(t) \le \dfrac{s}{N}\, |t|^{N/s} + |t|^{N^2/s(N-s)}\, e^{\, |t|^{N/(N-s)}}$,
		\item \label{Item 5} $F(t) \ge \ds{\frac{s}{N}\, |t|^{N/s} + \frac{s\, (N - s)}{N^2}\, |t|^{N^2/s(N-s)}}$.
	\end{enumroman}
\end{lemma}

\begin{proof}
	Since $f$ is odd, and hence $F$ is even,
	\[
	F(t) = \int_0^{|t|} f(\tau)\, d\tau = \int_0^{|t|} \tau^{(N-s)/s}\, e^{\, \tau^{N/(N-s)}}\, d\tau.
	\]
	
	\ref{Item 2} Integrating by parts,
	\begin{align*}
		F(t) & =  \frac{N - s}{N}\, |t|^{N/s-N/(N-s)}\, e^{\, |t|^{N/(N-s)}} - \frac{N - 2s}{s} \int_0^{|t|} \tau^{N/s-N/(N-s)-1}\, e^{\, \tau^{N/(N-s)}}\, d\tau\\[3pt]
		& \le \frac{N - s}{N}\, \frac{|t|^{N/s}\, e^{\, |t|^{N/(N-s)}}}{|t|^{N/(N-s)}}\\[3pt]
		& =  \frac{N - s}{N}\, \frac{t f(t)}{|t|^{N/(N-s)}}.
	\end{align*}
	
	\ref{Item 1} For $|t| \le 1$, $F(t) \le F(1)$. For $|t| > 1$, $F(t) = F(1) + \dint_1^{|t|} f(\tau)\, d\tau$. Integrating by parts,
	\begin{align*}
		\int_1^{|t|} f(\tau)\, d\tau & =  \frac{s}{N}\, |t|^{N/s}\, e^{\, |t|^{N/(N-s)}} - \frac{se}{N} - \frac{s}{N - s} \int_1^{|t|} \tau^{(N-s)/s+N/(N-s)}\, e^{\, \tau^{N/(N-s)}}\, d\tau\\[3pt]
		& \le  \frac{s}{N}\; t f(t) - \frac{s}{N - s} \int_1^{|t|} f(\tau)\, d\tau,
	\end{align*}
	and hence $\ds{\int_1^{|t|} f(\tau)\, d\tau \le \frac{s\, (N - s)}{N^2}\; t f(t)}$.
	
	\ref{Item 3} integrating by parts,
	\begin{align*}
		F(t) & =  \frac{s}{N}\, |t|^{N/s}\, e^{\, |t|^{N/(N-s)}} - \frac{s}{N - s} \int_0^{|t|} \tau^{N/s+N/(N-s)-1}\, e^{\, \tau^{N/(N-s)}}\, d\tau\\[3pt]
		& \le  \frac{s}{N}\; t f(t) - \frac{s}{N - s} \int_0^{|t|} \tau^{N^2/s(N-s)-1}\, d\tau\\[3pt]
		& =  \frac{s}{N}\; t f(t) - \frac{s^2}{N^2}\, |t|^{N^2/s(N-s)}.
	\end{align*}
	
	\ref{Item 4} Since $e^\tau \le 1 + \tau e^\tau$ for all $\tau \ge 0$,
	\begin{align*}
		F(t) & \le  \int_0^{|t|} \tau^{(N-s)/s} \left(1 + \tau^{N/(N-s)}\, e^{\, \tau^{N/(N-s)}}\right) d\tau\\[3pt]
		& \le  \int_0^{|t|} \tau^{(N-s)/s}\, d\tau + \int_0^{|t|} |t|^{(N-s)/s+N/(N-s)}\, e^{\, |t|^{N/(N-s)}}\, d\tau\\[3pt]
		& =  \frac{s}{N}\, |t|^{N/s} + |t|^{N^2/s(N-s)}\, e^{\, |t|^{N/(N-s)}}.
	\end{align*}
	
	\ref{Item 5} Since $e^\tau \ge 1 + \tau$ for all $\tau \ge 0$,
	\begin{align*}
		F(t) & \ge  \int_0^{|t|} \tau^{(N-s)/s} \left(1 + \tau^{N/(N-s)}\right) d\tau\\[3pt]
		& =  \frac{s}{N}\, |t|^{N/s} + \frac{s\, (N - s)}{N^2}\, |t|^{N^2/s(N-s)}. 
	\end{align*}
This concludes the proof.
\end{proof}

\section{Palais-Smale condition}

Recall that $\Phi$ satisfies the \PS{c} condition if every sequence $\seq{u_j}$ in $W^{s,N/s}_0(\Omega)$ such that $\Phi(u_j) \to c$ and $\Phi'(u_j) \to 0$, called a \PS{c} sequence, has a convergent subsequence. The main result of this section is the following theorem.

\begin{theorem} \label{Theorem 2}
	$\Phi$ satisfies the {\em \PS{c}} condition for all $c < \dfrac{s}{N}\; \alpha_{N,s}(\Omega)^{(N-s)/s}$.
\end{theorem}

First we prove a lemma.

\begin{lemma} \label{Lemma 2}
	If $u_j$ converges to $u$ weakly in $W^{s,N/s}_0(\Omega)$ and a.e.\! in $\Omega$, and
	\begin{equation} \label{5}
		\sup_{j\in\N} \int_\Omega u_j\, f(u_j)\, dx < \infty,
	\end{equation}
	then
	\[
	\int_\Omega F(u_j)\, dx \to \int_\Omega F(u)\, dx.
	\]
\end{lemma}

\begin{proof}
	For any $M > 0$, write
	\[
	\int_\Omega F(u_j)\, dx = \int_{\set{|u_j| < M}} F(u_j)\, dx + \int_{\set{|u_j| \ge M}} F(u_j)\, dx.
	\]
	By Lemma \ref{Lemma 1} \ref{Item 2} and \eqref{5}, we have
	\[
	\int_{\set{|u_j| \ge M}} F(u_j)\, dx \le \frac{N - s}{N M^{N/(N-s)}} \int_\Omega u_j\, f(u_j)\, dx = \O\! \left(\frac{1}{M^{N/(N-s)}}\right), \text{ as } M \to \infty.
	\]
	Hence
	\[
	\int_\Omega F(u_j)\, dx = \int_{\set{|u_j| < M}} F(u_j)\, dx + \O\! \left(\frac{1}{M^{N/(N-s)}}\right)\!,
	\]
	and the desired conclusion follows by letting $j \to \infty$ first and then $M \to \infty$.
\end{proof}

We are now ready to prove Theorem \ref{Theorem 2}.

\begin{proof}[Proof of Theorem \ref{Theorem 2}]
	Let $\seq{u_j}$ be a \PS{c} sequence. Then
	\begin{equation} \label{2}
		\Phi(u_j) = \frac{s}{N} \norm{u_j}^{N/s} - \lambda \int_\Omega F(u_j)\, dx = c + \o(1)
	\end{equation}
	and
	\begin{equation} \label{3}
		\Phi'(u_j)\, u_j = \norm{u_j}^{N/s} - \lambda \int_\Omega u_j\, f(u_j)\, dx = \o(\norm{u_j}).
	\end{equation}
	Since $s/N > s\, (N - s)/N^2$, it follows from Lemma \ref{Lemma 1} \ref{Item 1}, \eqref{2} and \eqref{3} that $\seq{u_j}$ is bounded in $W^{s,N/s}_0(\Omega)$. Hence a renamed subsequence converges to some $u$ weakly in $W^{s,N/s}_0(\Omega)$, strongly in $L^p(\Omega)$ for all $p \in [1,\infty)$, and a.e.\! in $\Omega$. Moreover,
	\[
	\sup_{j\in\N} \int_\Omega u_j\, f(u_j)\, dx < \infty
	\]
	by \eqref{3}, and hence
	\begin{equation} \label{4}
		\int_\Omega F(u_j)\, dx \to \int_\Omega F(u)\, dx
	\end{equation}
	by virtue of Lemma \ref{Lemma 2}.	
	By Lemma \ref{Lemma 1} \ref{Item 3}, \eqref{2}, and \eqref{3},
	\[
	\frac{\lambda s^2}{N^2} \int_\Omega |u_j|^{N^2/s(N-s)}\, dx \le \lambda \int_\Omega \bigg[\frac{s}{N}\; u_j\, f(u_j) - F(u_j)\bigg]\, dx = c + \o(1),
	\]
	so
	\[
	c \ge \frac{\lambda s^2}{N^2} \int_\Omega |u|^{N^2/s(N-s)}\, dx \ge 0.
	\]
	If $c = 0$, then $u = 0$ and hence $\dint_\Omega F(u_j)\, dx \to 0$ by \eqref{4}, so $\norm{u_j} \to 0$ by \eqref{2}.
	
	Now suppose that $0 < c < (s/N)\, \alpha_{N,s}(\Omega)^{(N-s)/s}$. We claim that the weak limit $u$ is nonzero. Suppose $u = 0$. Then
	\begin{equation} \label{6}
		\int_\Omega F(u_j)\, dx \to 0
	\end{equation}
	by \eqref{4} and hence
	\[
	\norm{u_j} \to \left(\frac{Nc}{s}\right)^{s/N} < \alpha_{N,s}(\Omega)^{(N-s)/N}
	\]
	by \eqref{2}. Let $(Nc/s)^{s/(N-s)} < \alpha < \alpha_{N,s}(\Omega)$. Then $\norm{u_j} \le \alpha^{(N-s)/N}$ for all $j \ge j_0$ for some $j_0$. Let $1 < q < \alpha_{N,s}(\Omega)/\alpha$. By the H\"{o}lder inequality,
	\[
	\int_\Omega u_j\, f(u_j)\, dx \le \left(\int_\Omega |u_j|^{Np/s}\, dx\right)^{1/p} \left(\int_\Omega e^{\, q\, |u_j|^{N/(N-s)}}\, dx\right)^{1/q},
	\]
	where $1/p + 1/q = 1$. The first integral on the right-hand side converges to zero since $u = 0$, while the second integral is bounded for $j \ge j_0$ since $q\, |u_j|^{N/(N-s)} = q\, \alpha\, |\widetilde{u}_j|^{N/(N-s)}$ with $q\, \alpha < \alpha_{N,s}(\Omega)$ and $\widetilde{u}_j = u_j/\alpha^{(N-s)/N}$ satisfies $\norm{\widetilde{u}_j} \le 1$, so
	\[
	\int_\Omega u_j\, f(u_j)\, dx \to 0.
	\]
	Then $u_j \to 0$ by \eqref{3}, and hence $c = 0$ by \eqref{2} and \eqref{6}, a contradiction. So $u$ is nonzero.
	
	Since $\Phi'(u_j) \to 0$,
	\[
	\int_{\R^{2N}} \frac{|u_j(x) - u_j(y)|^{(N-2s)/s}\, (u_j(x) - u_j(y))\, (v(x) - v(y))}{|x - y|^{2N}}\, dx dy - \lambda \int_\Omega f(u_j)\, v\, dx \to 0
	\]
	for all $v \in W^{s,N/s}_0(\Omega)$. For $v \in C^\infty_0(\Omega)$, an argument similar to that in the proof of Lemma \ref{Lemma 2} using the estimate
	\[
	\abs{\int_{\set{|u_j| \ge M}} f(u_j)\, v\, dx} \le \frac{\sup |v|}{M} \int_\Omega u_j\, f(u_j)\, dx = \O\! \left(\frac{1}{M}\right)
	\]
	shows that $\dint_\Omega f(u_j)\, v\, dx \to \dint_\Omega f(u)\, v\, dx$, so
	\[
	\int_{\R^{2N}} \frac{|u(x) - u(y)|^{(N-2s)/s}\, (u(x) - u(y))\, (v(x) - v(y))}{|x - y|^{2N}}\, dx dy = \lambda \int_\Omega f(u)\, v\, dx.
	\]
	Then this holds for all $v \in W^{s,N/s}_0(\Omega)$ by density, and taking $v = u$ gives
	\begin{equation} \label{8}
		\norm{u}^{N/s} = \lambda \int_\Omega u\, f(u)\, dx.
	\end{equation}
	
	Next we claim that
	\begin{equation} \label{12}
		\int_\Omega u_j\, f(u_j)\, dx \to \int_\Omega u\, f(u)\, dx.
	\end{equation}
	We have
	\begin{equation} \label{10}
		u_j\, f(u_j) = |u_j|^{N/s}\, e^{\, |u_j|^{N/(N-s)}} = |u_j|^{N/s}\, e^{\, \norm{u_j}^{N/(N-s)}\, |\widetilde{u}_j|^{N/(N-s)}},
	\end{equation}
	where $\widetilde{u}_j = u_j/\norm{u_j}$. By \eqref{2} and \eqref{4},
	\[
	\norm{u_j} \to \left[\frac{N}{s}\, (c + \lambda \beta)\right]^{s/N},
	\]
	where $\beta = \dint_\Omega F(u)\, dx$, so $\widetilde{u}_j$ converges a.e.\! to $\widetilde{u} = u/\left[(N/s)\, (c + \lambda \beta)\right]^{s/N}$. Then
	\[
	\norm{u_j}^{N/(N-s)}\, (1 - \norm{\widetilde{u}}^{N/s})^{s/(N-s)} \to \left[\frac{N}{s}\, (c + \lambda \beta) - \norm{u}^{N/s}\right]^{s/(N-s)} \le \left(\frac{Nc}{s}\right)^{s/(N-s)}
	\]
	since
	\[
	\norm{u}^{N/s} \ge \frac{\lambda N}{s} \int_\Omega F(u)\, dx = \frac{\lambda N \beta}{s}
	\]
	by \eqref{8} and Lemma \ref{Lemma 1} \ref{Item 3}. Let
	\[
	\frac{\left(\dfrac{Nc}{s}\right)^{s/(N-s)}}{(1 - \norm{\widetilde{u}}^{N/s})^{s/(N-s)}} < \alpha - 2 \eps < \alpha < \frac{\alpha_{N,s}(\Omega)}{(1 - \norm{\widetilde{u}}^{N/s})^{s/(N-s)}}.
	\]
	Then $\norm{u_j}^{N/(N-s)} \le \alpha - 2 \eps$ for all $j \ge j_0$ for some $j_0$, and
	\begin{equation} \label{11}
		\sup_{j\in\N} \int_\Omega e^{\, \alpha\, |\widetilde{u}_j|^{N/(N-s)}}\, dx < \infty
	\end{equation}
	by Theorem \ref{Theorem 1}. For $M > 0$ and $j \ge j_0$, \eqref{10} then gives
	\begin{align*}
		&  \int_{\set{|u_j| \ge M}} u_j\, f(u_j)\, dx\\[2pt]
		& \le  \int_{\set{|u_j| \ge M}} |u_j|^{N/s}\, e^{\, (\alpha - 2 \eps)\, |\widetilde{u}_j|^{N/(N-s)}}\, dx\\[2pt]
		& \le  \left(\max_{t \ge 0}\, t^{N/s}\, e^{\, - \eps\, t^{N/(N-s)}}\right) \norm{u_j}^{N/s} e^{\, - \eps\, (M/\norm{u_j})^{N/(N-s)}} \int_\Omega e^{\, \alpha\, |\widetilde{u}_j|^{N/(N-s)}}\, dx.
	\end{align*}
	The last expression goes to zero as $M \to \infty$ uniformly in $j$ since $\norm{u_j}$ is bounded and \eqref{11} holds, so \eqref{12} now follows as in the proof of Lemma \ref{Lemma 2}.	
	By \eqref{3}, \eqref{12}, and \eqref{8},
	\[
	\norm{u_j}^{N/s} \to \lambda \int_\Omega u\, f(u)\, dx = \norm{u}^{N/s}
	\]
	and hence $\norm{u_j} \to \norm{u}$, so $u_j \to u$ by the uniform convexity of $W^{s,N/s}_0(\Omega)$.
\end{proof}

\section{Eigenvalue problem}

The asymptotic problem associated with \eqref{1} as $u$ goes to zero is the eigenvalue problem
\begin{equation} \label{14}
	\left\{\begin{aligned}
		(- \Delta)_{N/s}^s\, u & = \lambda\, |u|^{(N-2s)/s}\, u && \text{in } \Omega\\
		u & = 0 && \text{in } \R^N \setminus \Omega.
	\end{aligned}\right.
\end{equation}
The weak formulation of this problem can be written as the operator equation
\begin{equation} \label{13}
	A(u) = \lambda\, B(u),
\end{equation}
where $A$ and $B$ are the nonlinear operators from $W^{s,N/s}_0(\Omega)$ to its dual $W^{-s,N/(N-s)}(\Omega)$ defined 
by setting
\begin{multline*}
	\langle A(u), v\rangle := \int_{\R^{2N}} \frac{|u(x) - u(y)|^{(N-2s)/s}\, (u(x) - u(y))\, (v(x) - v(y))}{|x - y|^{2N}}\, dx dy,\\[5pt]
	\langle B(u), v\rangle := \int_\Omega |u|^{(N-2s)/s}\, uv\, dx, \quad u, v \in W^{s,N/s}_0(\Omega),
\end{multline*}
respectively. The operators $A$ and $B$ are homogeneous of degree $(N - s)/s$, odd, and satisfy
\begin{equation*}
	\langle A(u), v\rangle \le \norm{u}^{(N-s)/s} \norm{v}, \quad \langle A(u), u\rangle = \norm{u}^{N/s},\quad
	\langle B(u), u\rangle = \pnorm[N/s]{u}^{N/s},\quad \forall u, v \in W^{s,N/s}_0(\Omega).
\end{equation*}
%where $\pnorm[N/s]{\cdot}$ denotes the norm in $L^{N/s}(\Omega)$. 

Since $W^{s,N/s}_0(\Omega)$ is uniformly convex, then $A$ is of type (S), i.e.\ every sequence $\seq{u_j}$ in $W^{s,N/s}_0(\Omega)$ such that $u_j \wto u$ and $\langle A(u_j), u_j - u\rangle \to 0$ as $j\to\infty$ has a subsequence that converges strongly to $u$
(see e.g.\ \cite[Proposition 1.3]{monogr}). Moreover, $B$ is a compact operator since the embedding 
$$
W^{s,N/s}_0(\Omega) \hookrightarrow L^{N/s}(\Omega),
$$ 
is compact. Hence, problem \eqref{13} falls into the abstract framework considered in \cite[Ch.\ 4]{monogr}
and we can construct an increasing and unbounded sequence of eigenvalues as follows.

Eigenvalues of problem \eqref{14} coincide with critical values of the functional
\[
\Psi(u) = \frac{1}{\pnorm[N/s]{u}^{N/s}}, \quad u \in \M = \bgset{u \in W^{s,N/s}_0(\Omega) : \norm{u} = 1}.
\]
Let $\F$ denote the class of symmetric subsets of $\M$, let $i(M)$ denote the $\Z_2$-cohomological index of $M \in \F$ (see Fadell and Rabinowitz \cite{fadell}), and set
\[
\lambda_k := \inf_{\substack{M \in \F\\[1pt] i(M) \ge k}}\, \sup_{u \in M}\, \Psi(u), \quad k \ge 1.
\]
Then
\[
\lambda_1 = \inf_{u \in \M}\, \Psi(u) > 0
\]
is the smallest eigenvalue and $\lambda_k \nearrow \infty$ is a sequence of eigenvalues (see \cite[Proposition 3.52]{monogr}). Moreover, denoting by
\[
\Psi^a := \set{u \in \M : \Psi(u) \le a}, \qquad \Psi_a := \set{u \in \M : \Psi(u) \ge a}
\]
the sub- and superlevel sets of $\Psi$, respectively, we have
\begin{equation} \label{17}
	i(\Psi^{\lambda_k}) = i(\M \setminus \Psi_{\lambda_{k+1}}) = k
\end{equation}
whenever $\lambda_k < \lambda_{k+1}$ (see \cite[Proposition 3.53]{monogr}). The main result of this section is the following.

\begin{theorem} \label{Theorem 3}
	If $\lambda_k < \lambda_{k+1}$, then the sublevel set $\Psi^{\lambda_k}$ contains a compact symmetric subset of index $k$.
\end{theorem}

First a couple of lemmas.

\begin{lemma} \label{Lemma 4}
	The operator $A$ is strictly monotone, i.e.,
	\[
	\langle A(u) - A(v), u - v\rangle  > 0
	\]
	for all $u \ne v$ in $W^{s,N/s}_0(\Omega)$.
\end{lemma}

\begin{proof}
	By \cite[Lemma 6.3]{monogr}, it suffices to show that
	\[
	\langle A(u), v\rangle \le \norm{u}^{(N-s)/s} \norm{v}, \quad \forall u, v \in W^{s,N/s}_0(\Omega)
	\]
	and the equality holds if and only if $\alpha u = \beta v$ for some $\alpha, \beta \ge 0$, not both zero. We have
	\[
	\langle A(u), v\rangle \le \int_{\R^{2N}} \frac{|u(x) - u(y)|^{(N-s)/s}\, |v(x) - v(y)|}{|x - y|^{2N}}\, dx dy \le \norm{u}^{(N-s)/s} \norm{v}
	\]
	by the H\"{o}lder inequality. Clearly, equality holds throughout if $\alpha u = \beta v$ for some $\alpha, \beta \ge 0$, not both zero. Conversely, 
	if $\langle A(u), v\rangle = \norm{u}^{(N-s)/s} \norm{v}$, equality holds in both inequalities. The equality in the second inequality gives
	\[
	\alpha\, |u(x) - u(y)| = \beta\, |v(x) - v(y)| \quad \text{a.e.\! in } \R^{2N}
	\]
	for some $\alpha, \beta \ge 0$, not both zero, and then the equality in the first inequality gives
	\[
	\alpha\, (u(x) - u(y)) = \beta\, (v(x) - v(y)) \quad \text{a.e.\! in } \R^{2N}.
	\]
	Since $u$ and $v$ vanish a.e.\! in $\R^N \setminus \Omega$, it follows that $\alpha u = \beta v$ a.e.\! in $\Omega$.
\end{proof}

\begin{lemma} \label{Lemma 5}
	For each $w \in L^{N/s}(\Omega)$, the problem
	\begin{equation} \label{15}
		\left\{\begin{aligned}
			(- \Delta)_{N/s}^s\, u & = |w|^{(N-2s)/s}\, w && \text{in } \Omega\\
			u & = 0 && \text{in } \R^N \setminus \Omega
		\end{aligned}\right.
	\end{equation}
	has a unique weak solution $u \in W^{s,N/s}_0(\Omega)$. Moreover, the map
	\[
	J : L^{N/s}(\Omega) \to W^{s,N/s}_0(\Omega), \quad w \mapsto u
	\]
	is continuous, homogeneous of degree $(N - s)/s$, and satisfies
	\begin{equation} \label{18}
		\frac{\norm{J(w)}}{\pnorm[N/s]{J(w)}} \le \frac{\norm{w}}{\pnorm[N/s]{w}}
	\end{equation}
	for all $w \ne 0$ in $L^{N/s}(\Omega)$.
\end{lemma}

\begin{proof}
	The existence follows from a standard minimization argument and the uniqueness from Lemma \ref{Lemma 4}.	
	Clearly, $J$ is homogeneous of degree $(N - s)/s$. To see that it is continuous, let $w_j \to w$ in $L^{N/s}(\Omega)$ and let $u_j = J(w_j)$, so
	\begin{equation} \label{16}
		\langle A(u_j), v\rangle = \int_\Omega |w_j|^{(N-2s)/s}\, w_j v\, dx \quad \forall v \in W^{s,N/s}_0(\Omega).
	\end{equation}
	Testing with $v = u_j$ gives
	\[
	\norm{u_j}^{N/s} = \int_\Omega |w_j|^{(N-2s)/s}\, w_j u_j\, dx \le \pnorm[N/s]{w_j}^{(N-s)/s} \pnorm[N/s]{u_j}
	\]
	by the H\"{o}lder inequality, which together with the imbedding $W^{s,N/s}_0(\Omega) \hookrightarrow L^{N/s}(\Omega)$ shows that $\seq{u_j}$ is bounded. 
	Therefore, a renamed subsequence of $\seq{u_j}$ converges to some $u$ weakly, strongly in $L^{N/s}(\Omega)$ and a.e.\! in $\Omega$. Then $u$ is a weak solution of problem \eqref{15} as in the proof of Theorem \ref{Theorem 2}, so $u = J(w)$. Testing \eqref{16} with $u_j - u$ gives
	\[
	\langle A(u_j), u_j - u\rangle = \int_\Omega |w_j|^{(N-2s)/s}\, w_j\, (u_j - u)\, dx \to 0,
	\]
	so $u_j \to u$ for a further subsequence since the operator $A$ is of type (S).	
	Finally, testing
	\[
	\langle A(u), v\rangle = \int_\Omega |w|^{(N-2s)/s}\, wv\, dx
	\]
	with $v = u, w$ and using the H\"{o}lder inequality gives
	\[
	\norm{u}^{N/s} \le \pnorm[N/s]{w}^{(N-s)/s} \pnorm[N/s]{u}, \qquad \pnorm[N/s]{w}^{N/s} \le \norm{u}^{(N-s)/s} \norm{w},
	\]
	from which \eqref{18} follows.
\end{proof}

We are now ready to prove Theorem \ref{Theorem 3}.

\begin{proof}[Proof of Theorem \ref{Theorem 3}]
	Let
	\[
	\pi(u) = \frac{u}{\norm{u}}, \quad \widetilde{\pi}(u) = \frac{u}{\pnorm[N/s]{u}}, \quad u \in W^{s,N/s}_0(\Omega) \setminus \set{0}
	\]
	be the radial projections onto $\M$ and 
	$$
	\widetilde{\M} = \bgset{u \in W^{s,N/s}_0(\Omega) : \pnorm[N/s]{u} = 1},
	$$
	respectively, let $i$ be the imbedding $W^{s,N/s}_0(\Omega) \hookrightarrow L^{N/s}(\Omega)$, let $J$ be the map defined in Lemma \ref{Lemma 5}, and let $\varphi : \Psi^{\lambda_k} \to \M$ be the composition of the maps
	\[
	\begin{CD}
	\Psi^{\lambda_k} @>\widetilde{\pi}>> \widetilde{\M} @>i>> L^{N/s}(\Omega) \setminus \set{0} @>J>> W^{s,N/s}_0(\Omega) \setminus \set{0} @>\pi>> \M.
	\end{CD}
	\]
	Since $i$ is compact,
	\[
	i(\widetilde{\pi}(\Psi^{\lambda_k})) = \bgset{u \in \widetilde{\M} : \norm{u}^{N/s} \le \lambda_k}
	\]
	is compact in $L^{N/s}(\Omega)$, and hence $K_0 = \varphi(\Psi^{\lambda_k})$ is compact in $W^{s,N/s}_0(\Omega)$. Since $\varphi$ is an odd continuous map, $i(K_0) \ge i(\Psi^{\lambda_k})$. For $u \in \Psi^{\lambda_k}$, $\varphi(u) = J(u)/\norm{J(u)}$ since $J$ is homogeneous, so
	\[
	\Psi(\varphi(u)) = \frac{\norm{J(u)}^{N/s}}{\pnorm[N/s]{J(u)}^{N/s}} \le \frac{\norm{u}^{N/s}}{\pnorm[N/s]{u}^{N/s}} = \Psi(u) \le \lambda_k
	\]
	by \eqref{18}, and hence $K_0 \subset \Psi^{\lambda_k}$. Then $i(K_0) \le i(\Psi^{\lambda_k})$ by the monotonicity of the index, so $i(K_0) = i(\Psi^{\lambda_k}) = k$ by \eqref{17}.
\end{proof}

\section{Bifurcation and multiplicity}
\label{bif-sect}

In this section we prove the following bifurcation and multiplicity results for problem \eqref{1}, in which the constant
\[
\mu_{N,s}(\Omega) = \alpha_{N,s}(\Omega)^{(N-s)/N}\! \left(\frac{N}{s\, \vol{\Omega}}\right)^{s/N}
\]
plays an important role, where $\vol{}$ denotes the Lebesgue measure in $\R^N$.

\begin{theorem} \label{Theorem 6}
	If
	\[
	\lambda < \lambda_1 < \lambda + \mu_{N,s}(\Omega)\, \lambda^{(N-s)/N},
	\]
	then problem \eqref{1} has a pair of nontrivial solutions $\pm\, u^\lambda$ such that $u^\lambda \to 0$ as $\lambda \nearrow \lambda_1$.
\end{theorem}

\begin{theorem} \label{Theorem 4}
	If $\lambda_k \le \lambda < \lambda_{k+1} = \cdots = \lambda_{k+m} < \lambda_{k+m+1}$ for some $k, m \ge 1$ and
	\begin{equation} \label{19}
		\lambda + \mu_{N,s}(\Omega)\, \lambda^{(N-s)/N} > \lambda_{k+1},
	\end{equation}
	then problem \eqref{1} has $m$ distinct pairs of nontrivial solutions $\pm\, u^\lambda_j,\, j = 1,\dots,m$ such that $u^\lambda_j \to 0$ as $\lambda \nearrow \lambda_{k+1}$.
\end{theorem}

In particular, we have the following existence result.

\begin{corollary}
	If
	\[
	\lambda_k \le \lambda < \lambda_{k+1} < \lambda + \mu_{N,s}(\Omega)\, \lambda^{(N-s)/N}
	\]
	for some $k \ge 1$, then problem \eqref{1} has a nontrivial solution.
\end{corollary}

\begin{remark}
	Since $\lambda \ge \lambda_k$ in Theorem \ref{Theorem 4}, \eqref{19} holds if
	\[
	\lambda > \lambda_{k+1} - \mu_{N,s}(\Omega)\, \lambda_k^{(N-s)/N},
	\]
	or if
	\[
	\lambda > \left(\frac{\lambda_{k+1} - \lambda_k}{\mu_{N,s}(\Omega)}\right)^{N/(N-s)}.
	\]
\end{remark}

We only give the proof of Theorem \ref{Theorem 4}. The proof of Theorem \ref{Theorem 6} 
is similar and simpler. The proof will be based on an abstract critical point theorem proved in 
Yang and Perera \cite{yype} that generalizes Bartolo et al.\ \cite[Theorem 2.4]{bartolo}. 

Let $\Phi$ be an even $C^1$-functional on a Banach space $W$. Let $\A^\ast$ denote the class of symmetric subsets of $W$, let $r > 0$, let $S_r = \set{u \in W : \norm{u} = r}$, let $0 < b \le \infty$, and let $\Gamma$ denote the group of odd homeomorphisms of $W$ that are the identity outside $\Phi^{-1}(0,b)$. The pseudo-index of $M \in \A^\ast$ related to $i$, $S_r$, and $\Gamma$ is defined by
\[
i^\ast(M) = \min_{\gamma \in \Gamma}\, i(\gamma(M) \cap S_r)
\]
(see Benci \cite{benci}).

\begin{theorem}[{\cite[Theorem 2.4]{yype}}] \label{Theorem 5}
	Let $K_0$ and $B_0$ be symmetric subsets of $\M = \set{u \in W : \norm{u} = 1}$ such that $K_0$ is compact, $B_0$ is closed, and
	\[
	i(K_0) \ge k + m, \qquad i(\M \setminus B_0) \le k
	\]
	for some $k \ge 0$ and $m \ge 1$. Assume that there exists $R > r$ such that
	\[
	\sup \Phi(K) \le 0 < \inf \Phi(B), \qquad \sup \Phi(X) < b,
	\]
	where $K = \set{Ru : u \in K_0}$, $B = \set{ru : u \in B_0}$, and $X = \set{tu : u \in K,\, 0 \le t \le 1}$. For $j = k + 1,\dots,k + m$, let
	\[
	\A_j^\ast = \set{M \in \A^\ast : M \text{ is compact and } i^\ast(M) \ge j}
	\]
	and set
	\[
	c_j^\ast := \inf_{M \in \A_j^\ast}\, \max_{u \in M}\, \Phi(u).
	\]
	Then
	\[
	\inf \Phi(B) \le c_{k+1}^\ast \le \dotsb \le c_{k+m}^\ast \le \sup \Phi(X),
	\]
	in particular, $0 < c_j^\ast < b$. If, in addition, $\Phi$ satisfies the {\em \PS{c}} condition for all $c \in (0,b)$, then each $c_j^\ast$ is a critical value of $\Phi$ and there are $m$ distinct pairs of associated critical points.
\end{theorem}

We are now ready to prove Theorem \ref{Theorem 4}.

\begin{proof}[Proof of Theorem \ref{Theorem 4}]
	In view of Theorem \ref{Theorem 2}, we apply Theorem \ref{Theorem 5} with
	\[
	b := \frac{s}{N}\; \alpha_{N,s}(\Omega)^{(N-s)/s}.
	\]
	By Theorem \ref{Theorem 3}, the sublevel set $\Psi^{\lambda_{k+m}}$ has a compact symmetric subset $K_0$ with
	\[
	i(K_0) = k + m.
	\]
	We take $B_0 := \Psi_{\lambda_{k+1}}$, so that
	\[
	i(\M \setminus B_0) = k
	\]
	by \eqref{17}. Let $R > r > 0$ and let $K$, $B$, and $X$ be as in Theorem \ref{Theorem 5}.	
	By Lemma \ref{Lemma 1} \ref{Item 4},
	\[
	\Phi(u) \ge \frac{s}{N} \left(\norm{u}^{N/s} - \lambda \int_\Omega |u|^{N/s}\, dx\right) - \lambda \int_\Omega |u|^{N^2/s(N-s)}\, e^{\, |u|^{N/(N-s)}}\, dx,
	\]
	so for $u \in \Psi_{\lambda_{k+1}}$,
	\begin{align*}
		\Phi(ru) & \ge\  \frac{s r^{N/s}}{N} \left(1 - \frac{\lambda}{\Psi(u)}\right) - \lambda r^{N^2/s(N-s)} \int_\Omega |u|^{N^2/s(N-s)}\, e^{\, r^{N/(N-s)}\, |u|^{N/(N-s)}}\, dx\\[10pt]
		& \ge  r^{N/s}\, \Bigg[\frac{s}{N} \left(1 - \frac{\lambda}{\lambda_{k+1}}\right)\\[5pt]
		&  - \lambda r^{N/(N-s)} \left(\int_\Omega |u|^{2N^2/s(N-s)}\, dx\right)^{1/2} \left(\int_\Omega e^{\, 2r^{N/(N-s)}\, |u|^{N/(N-s)}}\, dx\right)^{1/2}\Bigg].
	\end{align*}
	The first integral in the last expression is bounded since $W^{s,N/s}_0(\Omega) \hookrightarrow L^{2N^2/s(N-s)}(\Omega)$, and the second integral is also bounded if $2r^{N/(N-s)} < \alpha_{N,s}(\Omega)$. Since $\lambda < \lambda_{k+1}$, it follows that $\inf \Phi(B) > 0$ if $r$ is sufficiently small.
	By Lemma \ref{Lemma 1} \ref{Item 5} and the H\"{o}lder inequality,
	\begin{align*}
		\Phi(u) & \le  \frac{s}{N} \norm{u}^{N/s} - \frac{\lambda s\, (N - s)}{N^2} \int_\Omega |u|^{N^2/s(N-s)}\, dx\\[5pt]
		& \le  \frac{s}{N} \norm{u}^{N/s} - \frac{\lambda s\, (N - s)}{N^2\, \vol{\Omega}^{s/(N-s)}} \left(\int_\Omega |u|^{N/s}\, dx\right)^{N/(N-s)},
	\end{align*}
	so for $u \in K_0 \subset \Psi^{\lambda_{k+1}}$,
	\begin{align*}
		\Phi(Ru) & \le  \frac{s R^{N/s}}{N} - \frac{\lambda s\, (N - s)\, R^{N^2/s(N-s)}}{N^2\, \vol{\Omega}^{s/(N-s)}\, \Psi(u)^{N/(N-s)}}\\[10pt]
		& \le  - \frac{s R^{N/s}}{N} \left(\frac{\lambda\, (N - s)\, R^{N/(N-s)}}{\lambda_{k+1}^{N/(N-s)}\, N\, \vol{\Omega}^{s/(N-s)}} - 1\right).
	\end{align*}
	It follows that $\Phi \le 0$ on $K$ if $R$ is sufficiently large.
	By Lemma \ref{Lemma 1} \ref{Item 5} and the H\"{o}lder inequality,
	\begin{align*}
		\Phi(u) & \le  \frac{s}{N} \norm{u}^{N/s} - \lambda \int_\Omega \left[\frac{s}{N}\, |u|^{N/s} + \frac{s\, (N - s)}{N^2}\, |u|^{N^2/s(N-s)}\right] dx\\[5pt]
		& \le  \frac{s}{N} \left(\norm{u}^{N/s} - \lambda \int_\Omega |u|^{N/s}\, dx\right) - \frac{\lambda s\, (N - s)}{N^2\, \vol{\Omega}^{s/(N-s)}} \left(\int_\Omega |u|^{N/s}\, dx\right)^{N/(N-s)},
	\end{align*}
	so for $u \in X$,
	\begin{align*}
		\Phi(u) & \le  \frac{(\lambda_{k+1} - \lambda)\, s}{N} \int_\Omega |u|^{N/s}\, dx - \frac{\lambda s\, (N - s)}{N^2\, \vol{\Omega}^{s/(N-s)}} \left(\int_\Omega |u|^{N/s}\, dx\right)^{N/(N-s)}\\[10pt]
		& \le  \sup_{\rho \ge 0}\, \left[\frac{(\lambda_{k+1} - \lambda)\, s \rho}{N} - \frac{\lambda s\, (N - s)\, \rho^{N/(N-s)}}{N^2\, \vol{\Omega}^{s/(N-s)}}\right]\\[10pt]
		& =  \frac{(\lambda_{k+1} - \lambda)^{N/s}\, s^2\, \vol{\Omega}}{\lambda^{(N-s)/s}\, N^2}.
	\end{align*}
	So
	\[
	\sup \Phi(X) \le \frac{(\lambda_{k+1} - \lambda)^{N/s}\, s^2\, \vol{\Omega}}{\lambda^{(N-s)/s}\, N^2} < \frac{s}{N}\; \alpha_{N,s}(\Omega)^{(N-s)/s}
	\]
	by \eqref{19}. Thus, problem \eqref{1} has $m$ distinct pairs of nontrivial solutions $\pm\, u^\lambda_j,\, j = 1,\dots,m$ such that
	\begin{equation} \label{21}
		0 < \Phi(u^\lambda_j) \le \frac{(\lambda_{k+1} - \lambda)^{N/s}\, s^2\, \vol{\Omega}}{\lambda^{(N-s)/s}\, N^2}
	\end{equation}
	by Theorem \ref{Theorem 5}.	
	To prove that $u^\lambda_j \to 0$ as $\lambda \nearrow \lambda_{k+1}$, it suffices to show that for every sequence $\nu_n \nearrow \lambda_{k+1}$, a subsequence of $v_n := u^{\nu_n}_j$ converges to zero. We have
	\begin{equation} \label{22}
		\Phi(v_n) = \frac{s}{N} \norm{v_n}^{N/s} - \nu_n \int_\Omega F(v_n)\, dx \to 0
	\end{equation}
	by \eqref{21} and
	\begin{equation} \label{23}
		\Phi'(v_n)\, v_n = \norm{v_n}^{N/s} - \nu_n \int_\Omega v_n\, f(v_n)\, dx = 0.
	\end{equation}
	Since $s/N > s\, (N - s)/N^2$, it follows from Lemma \ref{Lemma 1} \ref{Item 1}, \eqref{22}, and \eqref{23} that $\seq{v_n}$ is bounded in $W^{s,N/s}_0(\Omega)$. Hence a renamed subsequence converges to some $v$ weakly in $W^{s,N/s}_0(\Omega)$, strongly in $L^p(\Omega)$ for all $p \in [1,\infty)$, and a.e.\! in $\Omega$. By Lemma \ref{Lemma 1} \ref{Item 3}, \eqref{22}, and \eqref{23},
	\[
	\frac{s^2}{N^2} \int_\Omega |v_n|^{N^2/s(N-s)}\, dx \le \int_\Omega \bigg[\frac{s}{N}\; v_n\, f(v_n) - F(v_n)\bigg]\, dx = \frac{\Phi(v_n)}{\nu_n} \le \frac{\Phi(v_n)}{\lambda_k} \to 0,
	\]
	so
	\[
	\int_\Omega |v|^{N^2/s(N-s)}\, dx = 0
	\]
	and hence $v = 0$. Since $\int_\Omega v_n\, f(v_n)\, dx$ is bounded by \eqref{23}, then 
	$$
	\int_\Omega F(v_n)\, dx \to 0,
	$$ 
	by Lemma \ref{Lemma 2}, so $\norm{v_n} \to 0$ by \eqref{22}.
\end{proof}

\bigskip

\end{document}